\documentclass[11pt,oneside,english]{amsart}
\usepackage[T1]{fontenc}
\usepackage[latin9]{inputenc}
\usepackage{geometry}
\usepackage{color}
\usepackage{mathrsfs}
\usepackage{amsthm}
\usepackage{amssymb}
\usepackage{setspace}
\usepackage{esint}
\makeatletter
\numberwithin{equation}{section}
\numberwithin{figure}{section}
 \theoremstyle{definition}
 \newtheorem*{defn*}{\protect\definitionname}
\theoremstyle{plain}
\newtheorem{thm}{\protect\theoremname}[section]
  \theoremstyle{plain}
  \newtheorem{fact}[thm]{\protect\factname}
  \theoremstyle{plain}
  \newtheorem{lem}[thm]{\protect\lemmaname}
  \theoremstyle{plain}
  \newtheorem{cor}[thm]{\protect\corollaryname}
  \theoremstyle{remark}
  \newtheorem*{rem*}{\protect\remarkname}
  \theoremstyle{remark}
  \newtheorem{rem}[thm]{\protect\remarkname}
  \theoremstyle{plain}
  \newtheorem{prop}[thm]{\protect\propositionname}

\makeatother

\usepackage{babel}
  \providecommand{\corollaryname}{Corollary}
  \providecommand{\definitionname}{Definition}
  \providecommand{\factname}{Fact}
  \providecommand{\lemmaname}{Lemma}
  \providecommand{\propositionname}{Proposition}
  \providecommand{\remarkname}{Remark}
\providecommand{\theoremname}{Theorem}

\begin{document}
\global\long\def\BB{\mathcal{B}}

\global\long\def\RR{\mathbb{R}}

\global\long\def\Tau{\mathcal{T}}

\global\long\def\Ra{\mathcal{R}}

\global\long\def\NN{\mathbb{N}}

\global\long\def\ZZ{\mathbb{Z}}

\global\long\def\P{\mathrm{P}}

\title{On the $K$ property for Maharam extensions of Bernoulli shifts and
a question of Krengel}

\author{Zemer Kosloff}

\curraddr{School of Mathematical Sciences, Tel Aviv University, 69978 Tel Aviv,
Israel. Email: zemerkos@post.tau.ac.il}
\begin{abstract}
We show that the Maharam extension of a type ${\rm III}$, conservative
and non singular $K$ Bernoulli is a $K$-transformation. This together
with the fact that the Maharam extension of a conservative transformation
is conservative gives a negative answer to Krengel's and Weiss's questions
about existence of a type ${\rm II}_{\infty}$ or type ${\rm III}_{\lambda}$
with $\lambda\neq1$ Bernoulli shift. A conservative non singular
$K$, in the sense of Silva and Thieullen, Bernoulli shift is either
of type ${\rm II}_{1}$ or of type ${\rm III}_{1}$. 
\end{abstract}
\begin{singlespace}

\thanks{This research was was supported by THE ISRAEL SCIENCE FOUNDATION
grant No. 1114/08.}
\end{singlespace}

\maketitle

\section{Introduction}

Let $T$ be an invertible non singular transformation of the probability
space $\left(X,\BB,\mu\right)$. The Maharam extension $\tilde{T}$
of $T$ is a measure preserving transformation which is a skew product
extension of $T$ with the Radon Nykodym cocycle. It is well known
that the Maharam extension is ergodic if and only if $T$ is of Krieger
type ${\rm III}_{1}$, see below. Here we show that in the case when
$T$ is a conservative non singular Bernoulli shift which satisfies
the $K$-property as in\cite{ST} with the one sided shift as the
exact factor , then the Maharam extension is a $K$-transformation.
Thus the Maharam extension is weak mixing in the sense that $T\times S$
is ergodic for every ergodic probability preserving transformation
$S$ and it has a countable Lebesgue spectrum. 

This type of non singular Bernoulli shifts was considered first in
\cite{Kre} where a shift without an absolutely continuous invariant
probability was constructed. Later Hamachi in \cite{Ham} constructed
an ergodic shift without an absolutely continuous $\sigma$-finite
invariant measure. Such transfomrations are called \textit{type} ${\rm III}$.
Krengel \cite{Kre} asked the question whether there exists a shift
with an absolutely continuous invariant $\sigma$-finite measure but
no such probability (these are called\textit{ type} ${\rm II}_{\infty}$).
The type ${\rm III}$ transformations can be further classified into
orbit equivalence classes according to their ratio set. In \cite{Kos}
a Bernoulli shift which is of Krieger type ${\rm III}_{1}$ was constructed.
In a presentation of that result Benjy Weiss asked whether there are
type ${\rm III}$ shifts of different Krieger types. As a corollary
of the $K$ property of the Maharam extension we get a dichotomy.
Namely an ergodic non singular $K$ Bernoulli shift is either of type
${\rm II}_{1}$ when the measure is equivalent to a stationary product
measure or of type ${\rm III}_{1}$. 

The proof makes use of the fact that since the Radon Nykodym cocycle
is measurable with respect to the $\sigma$-algebra $\BB_{\{0,1\}^{\NN}}$,
the Maharam extension is the natural extension of a skew product $\sigma_{\varphi}$
of the one sided shift. Thus it is enough to show that the tail equivalence
relation of the non invertible skew product is ergodic. This is done
by showing that the tail equivalence relation of $\sigma_{\varphi}$
is the orbit equivalence relation of the Maharam extension of the
odometer and proving that the odometer with the one sided measure
is of type ${\rm III}_{1}$. 

One step in the proof that the corresponding odometer action is type
${\rm III}_{1}$ is to show that for shift conservative product measures
we have two subsequences $n_{k}\to\infty$ and $m_{k}\to-\infty$
for which 
\[
\lim_{k\to\infty}P_{n_{k}}=\lim_{k\to\infty}P_{m_{k}}.
\]
The question arises whether for conservative shifts the limit needs
to exist? We give an example of a conservative shift with 
\[
\liminf_{k\to\infty}P_{k}(0)<\limsup_{k\to\infty}P_{k}(0),
\]
thus answering this question on the negative. 

\begin{singlespace}
\textit{\textcolor{black}{Acknowledgments.}}\textcolor{red}{{} }\textcolor{black}{I
would like to thank my advisor Prof. Jon Aaronson for many valuable
suggestions. I would also like to thank Prof. Ulrich Krengel for sending
me a copy of Michael Grewe's master thesis. }
\end{singlespace}

\section{Preliminaries}

\subsection{Non Singular Ergodic Theory}

Let $\left(X,\BB,\mu\right)$ be a standard measure space. Since one
can always pass to an equivalent probability measure, we will always
assume that $\mu$ is a probability measure. In what follows all equalities
of sets are modulo the measure on the space. 

A measurable transformation $T:X\to X$ is non singular if $\mu$
and $T_{*}\mu=\mu\circ T^{-1}$ are equivalent, meaning that they
have the same collection of null sets. In the case when $T$ is invertible
there exists the Radon Nykodym derivatives 
\[
T^{n'}(x):=\frac{d\mu\circ T^{n}}{d\mu}(x).
\]
When $T_{*}\mu=\mu$ we say that $T$ is $\mu$ preserving or $\mu$
is $T$ invariant. A transformation is \textit{ergodic} if $T^{-1}A=A$
implies $A\in\left\{ \emptyset,X\right\} $. A set $A\in\BB$ is \textit{wandering}
if $\left\{ T^{-n}A\right\} _{n=1}^{\infty}$ are disjoint. Denote
by $\mathfrak{D}$ the (measurable) union of all wandering sets, it's
complement is denoted by $\mathfrak{C}$ and is called the conservative
part. In the case where $\mathfrak{D}=X$ we say that $T$ is \textit{dissipative}.
If $\mathfrak{C}=X$ we say that $T$ is \textit{conservative}. By
Hopf's theorem \cite[Prop. 1.3.1.]{Aa} 
\begin{eqnarray}
\mathfrak{D} & = & \left\{ x\in X:\ \sum_{n=1}^{\infty}T^{n'}(x)<\infty\right\} \label{eq: Conservative Dissipative}\\
\mathfrak{C} & = & \left\{ x\in X:\ \sum_{n=1}^{\infty}T^{n'}(x)=\infty\right\} .\nonumber 
\end{eqnarray}

An invertible transformation $T$ satisfies the $K$\textit{-property}
if there exists a sub-$\sigma$ algebra $\mathcal{F}\subset\BB$ such
that $T^{-1}\mathcal{F\subset F}$, ${\displaystyle \bigcap_{n\in\ZZ}T^{n}\mathcal{F}=\left\{ \emptyset,X\right\} }$
and $\vee_{n=1}^{\infty}T^{-n}\mathcal{F}=\BB$. If $T$ is measure
preserving and $K$ then $T$ is either conservative or totally dissipative.
This property remains true in the case of non singular Bernoulli shifts,
see Lemma \ref{lem:[Grewe]} or \cite{Gre}. 

A measure preserving transformation $\left(Y,\BB_{Y},\nu,S\right)$
is an extension of $\left(X,\BB_{X},\mu,T\right)$ (equivalently $T$
is a factor of $X$) if there exists a measurable map $\pi:Y\to X$
such that $\pi^{-1}\BB_{X}\subset\BB_{Y}$, $\pi\circ S=T\circ\pi$
and $\pi_{*}\nu=\mu$. Given a non-invertible measure preserving transformation
$\left(X,\BB_{X},\mu,T\right)$, the \textit{natural extension} of
$T$ is an invertible measure preserving transformation $\check{T}$
which is minimal in the sense that 
\[
\vee_{n=1}^{\infty}\check{T}^{n}\pi^{-1}\BB_{X}=\BB_{\check{X}},
\]
where $\pi:\check{X}\to X$ is the factor map.

\subsection{Cocycles and skew product extensions}

A function $\varphi:\mathbb{N}\times X\to\RR$ ( or $\ZZ\times X\to\RR$
when $T$ is invertible) is a \textit{cocycle} if for every $n,m\in\NN$
and almost every $x\in X$,
\begin{equation}
\varphi_{n+m}(x)=\varphi_{n}(x)+\varphi_{m}\left(T^{n}x\right).\label{eq:cocycle equations}
\end{equation}
Given a function $\varphi:X\to\RR$ we can define the cocycle 
\[
\forall n\in\NN,\ \varphi_{n}(x)=\varphi(x)+\varphi\circ T(x)+\cdots+\varphi\circ T^{n-1}(x),
\]
and the \textit{skew product extension} $T_{\varphi}:\left(X\times\RR,\BB_{X}\otimes\BB_{\RR},\mu\times e^{s}ds\right)$
of $T$ with $\varphi$ by
\[
T_{\varphi}(x,y):=\left(Tx,y+\varphi(x)\right).
\]
 
\begin{defn*}
The set of \textit{essential values} for $\varphi$ is 
\[
e(T,\varphi)=\left\{ t\in\RR:\ \forall\epsilon>0,\forall A\in\left(\BB_{X}\right)_{+},\exists n\in\NN\ s.t.\ \mu\left(A\cap T^{-n}A\cap\left[\left|\varphi_{n}-a\right|<\epsilon\right]\right)>0\right\} 
\]

\end{defn*}
It follows from the cocycle equation \eqref{eq:cocycle equations}
that the set of essential values is a closed subset (under addition)
of $\RR$ and therefore it is of the form $\emptyset,\{0\},\{0\}\cup a\ZZ\ (a\in\RR)$
or $\RR$. The skew product $T_{\varphi}$ is ergodic if and only
if $T$ is ergodic and $e(S,\varphi)=\RR$. 

We will be interested in the \textit{Maharam extension} $\tilde{T}$
which is the skew product extension of an invertible transformation
$T:\left(X,\BB,\mu\right)\circlearrowleft$ with $\varphi(x)=\log T'(x)$,
the Radon-Nykodym cocycle. In the case when the Maharam extension
is ergodic we say that $T$ is of type ${\rm III}_{1}$ ($e\left(T,\log\frac{d\mu\circ S}{d\mu}\right)=\RR$
and $T$ is ergodic). In the case where $T$ is conservative and there
exists a $\mu-$ equivalent $\sigma-$finite invariant measure the
essential value set is $e\left(T,\log\frac{d\mu\circ S}{d\mu}\right)=\{0\}$.

\subsection{The tail and the orbital equivalence relation of a transformation}

For a more detailed discussion of the contents of this subsection
see \cite{KM}. 

Let $\left(X,\BB_{X}\right)$ be a standard measure space. An equivalence
relation on $X$ is a set $\Ra\subset X\times X$ such that the relation
$x\sim y$ if and only if $(x,y)\in\Ra$ is an equivalence relation.
It is measurable if $\Ra\subset\BB_{X}\otimes\BB_{X}$. Given an equivalence
relation $\Ra$ and a set $A\in\BB_{X}$, the \textit{saturation}
of $A$ is the set
\[
\Ra(A):=\bigcup_{x\in A}R_{x},
\]
where $\Ra_{x}:=\left\{ y\in X:\ (x,y)\in R\right\} .$ Given a measure
$\mu$ on $X$, we say that $\Ra$\textit{ is $\mu-$ergodic} if for
each $A\in\BB_{X}$, 
\[
\Ra(A)\in\left\{ \emptyset,X\right\} mod\ m.
\]

An equivalence relation is finite (respectively countable) if for
all $x\in X$, $\Ra_{x}$ is a finite (countable) set. It is hyperfinite
if there exists an increasing sequence of finite subequivalence relation
$E_{1}\subset E_{2}\subset\cdots\subset\Ra$ such that
\[
\Ra=\bigcup_{n=1}^{\infty}E_{n}.
\]

Given a non singular non-invertible transformation $\left(X,\BB_{X},\nu,S\right)$
we define the \textit{orbit equivalence relation} on $X\times X$

\[
\Ra_{S}:=\left\{ \left(y_{1},y_{2}\right)\in X\times X:\ \exists n,m\in\mathbb{N},\ S^{n}y_{1}=S^{m}y_{2}\right\} .
\]
and the \textit{tail relation}, which we denote by $\Tau(S)$, by
\[
\Tau(S)=\left\{ \left(y_{1},y_{2}\right):\ \exists n\in\NN,\ S^{n}y_{1}=S^{n}y_{2}\right\} .
\]
A transformation is \textit{exact} if for all $A\in\BB$, 
\[
\Tau(S)_{A}\in\{\emptyset,Y\}\ mod\nu
\]
By \cite{We,SlS} an equivalence relation is hyperfinite if and only
if it is an orbit relation of a non singular transformation. Therefore
if $\Tau_{S}$ is hyperfinite, which is true in our setting since
the shift is finite to one, there exists a non-singular transformation
$V$ of $\left(Y,\BB_{Y},\nu\right)$, which we call the \textit{tail
action of} $S$, such that 
\[
\Ra_{V}=\Tau_{S}.
\]
 It follows that $S$ is exact if and only if $V$ is ergodic. 

A function $\hat{\varphi}:\Ra\to\RR$ is an \textit{orbital cocycle}
if for every $x,y,z\in X$ in the same equivalence class of $\Ra$,
\[
\hat{\varphi}(x,y)=\hat{\varphi}(x,z)+\hat{\varphi}(z,y).
\]
To every function $\varphi:X\to\RR$ corresponds an orbital cocycle
$\hat{\varphi}$ on $\Tau_{S}$ (notice that the sum is actually a
finite sum) defined by 
\[
\hat{\varphi}\left(y_{1},y_{2}\right):=\sum_{n=0}^{\infty}\left\{ \varphi\left(S^{n}y_{1}\right)-\varphi\left(S^{n}y_{2}\right)\right\} ,\ \left(y_{1},y_{2}\right)\in\Tau(S).
\]
and the $\Ra_{V}$-cocycle $\psi$ defined by
\[
\psi(y)=\hat{\varphi}\left(y,Vy\right).
\]

The following fact shows that the skew product $S_{\varphi}$ is exact
if and only if $V_{\psi}$ is ergodic where $V$ is the tail action
of $S$ and $\psi$ is its corresponding cocycle. 
\begin{fact}
\cite{ANS}\label{pro:tail of skew product}Let $\left(Y,\BB_{Y},\nu,S\right)$
be a non singular and non-invertible transformation and $\left(Y,\BB_{Y},\nu,V\right)$
its associated tail action. Let $\varphi:Y\to\RR$ be a function and
$\psi$ the corresponding $\Ra_{V}$ cocycle. Then 
\[
\Tau_{S_{\varphi}}=\Ra_{V_{\psi}}.
\]

\end{fact}

\subsection{The Zero Type property and dissipative transformations:}

Given two measures on $\left(X,\BB\right)$ we can define the Hellinger
Integral \cite{Kak,Kos} by 
\[
\rho\left(\mu,\nu\right)=\int_{X}\sqrt{\frac{d\mu}{d\lambda}}\sqrt{\frac{d\nu}{d\lambda}}d\lambda
\]
where $\lambda$ is any measure on $X$ such that $\nu\ll\lambda$
and $\mu\ll\lambda$. 

If $T$ is a non singular transformation of $\left(X,\BB,\mu\right)$
then since $T_{*}^{n}\mu\sim\mu$ we have 
\[
\rho(n):=\rho\left(\mu,T_{*}^{n}\mu\right)=\int_{X}\sqrt{T^{n'}(x)}d\mu(x).
\]
A transformation is Zero-Type (sometimes also called mixing) if the
maximal spectral type of its Koopman operator defined by 
\[
\forall f\in L^{2}(X,\mu),\ U_{T}f:=\sqrt{T'}\cdot f\circ T
\]
is a Rajchman measure. This is equivalent to the condition: For every
$f,g\in L^{2}(X,\mu)$,
\[
\int_{X}U_{T}^{n}f\cdot\bar{g}d\mu\xrightarrow[n\to\infty]{}0.
\]
Note that when $T$ is probability preserving one needs to restrict
the class of functions to $L^{2}(X,\mu)\ominus\mathbb{C}$. 

The next lemma will be used to get a necessary criterion for conservativity
of Bernoulli shifts. .
\begin{lem}
\label{lem:Dissipative zero type}If $\left(X,\BB,\mu,T\right)$ is
zero type and ${\displaystyle \sum_{n=1}^{\infty}\rho\left(\mu,T_{*}^{n}\mu\right)<\infty}$
then $T$ is dissipative. \end{lem}
\begin{proof}
Since 
\[
\mu\left(\left|T^{n'}\right|>1\right)\leq\int_{X}\sqrt{T^{n'}}d\mu=\rho\left(\mu,T_{*}^{n}\mu\right)
\]
and the right hand side is summable, it follows from the Borel Cantelli
lemma that for almost every $x\in X$ there exists $N(x)\in\NN$ such
that for every $n>N(x)$, 
\[
T^{n'}(x)\leq\sqrt{T^{n'}(x)}\leq1.
\]
In addition the summability condition on $\rho\left(\mu,T_{*}^{n}\mu\right)$
ensures that 
\[
\sum_{n=1}^{\infty}\sqrt{T^{n'}}<\infty\ a.e\ d\mu.
\]
 Therefore by comparison of sums we have that 
\[
\sum_{n=1}^{\infty}T^{n'}(x)<\infty\ a.e.\ d\mu
\]
and so $T$ is dissipative. 
\end{proof}

\section{Half stationary Bernoulli Shifts}

\subsection{Non Singular Bernoulli Shift}

Let $X=\{0,1\}^{\mathbb{Z}},\BB=\BB_{X}$ , $X^{+}=\{0,1\}^{\mathbb{N}}$
and $\BB^{+}=\BB_{X^{+}}$. We will write $\sigma$ for the one-sided
shift on $X^{+}$ and $T$ for the full shift on $X$.

A product measure $P={\displaystyle \prod_{k=-\infty}^{\infty}P_{k}}\in\mathcal{P}\left(X\right)$
is \textbf{\textit{half stationary}}\texttt{ }if there exists $p\in(0,1)$
such that for all $k\leq0$,
\[
P_{k}\left(0\right)=1-P_{k}(1)=p.
\]

We will consider the case $p=\frac{1}{2}$. The case of general $p$
being similar. 

Thus the general form of a half stationary product measure (with $p=\frac{1}{2}$)
is 
\begin{equation}
P_{k}(0)=1-P_{k}(1)=\begin{cases}
\frac{1-a_{i}}{2} & k\in\NN\\
\frac{1}{2} & k\leq0
\end{cases},\label{Half stationary product measure}
\end{equation}
where $a_{i}\in\left(-1,1\right)$.

Let $P^{+}={\displaystyle \prod_{k=1}^{\infty}P_{k}}$ denote the
measure of $P$ restricted to $X^{+}$. If $P$ is half stationary,
then the full shift $T$ is the natural extension ,in the sense of
Silva and Thieullen \cite{ST}, of the one sided shift $\left(X^{+},\BB,P^{+}={\displaystyle \prod_{k=1}^{\infty}P_{k},\sigma}\right)$.
Since by Kolmogorov's $0-1$ Law the one sided shift is exact, the
full shift is a $K$-transformation. Conversely every $K$-Bernoulli
shift such that $T'$ is $\BB^{+}$ measurable is a shift with a half
stationary measure. We call such transformations non-singular $K$-shifts. 

The following gives conditions on the product measures so that the
shift is non singular and ergodic. 

\begin{thm} \label{thm: Kakutani}Let $P$ be of the form \eqref{Half stationary product measure}.
Then\begin{enumerate}

\item The shift $\left(X,\BB,P,T\right)$ is non singular if and
only if for all $n\in\NN$, $\left|a_{n}\right|\neq1$ and 
\begin{equation}
\sum_{k=0}^{\infty}\left\{ \left(\sqrt{P_{k}(0)}-\sqrt{P_{k+1}(0)}\right)^{2}+\left(\sqrt{P_{k}(1)}-\sqrt{P_{k+1}(1)}\right)^{2}\right\} <\infty.\label{eq:non singular condition}
\end{equation}

\item \label{ext:(II.2)-Radon Nykodym derivatives}For every $n\in\mathbb{N}$,
\[
T^{n'}(x)=\frac{dP\circ T^{n}}{dP}(x)=\prod_{k=1}^{\infty}\frac{P_{k-n}\left(w_{k}\right)}{P_{k}\left(w_{k}\right)}.
\]

\item\label{ext:(3)-Conservative implies ergodic}If the shift is
conservative then it is ergodic. 

\item There is an absolutely continuous invariant probability if
and only 
\[
\sum_{k=1}^{\infty}a_{k}^{2}<\infty.
\]

\item There exists constants $c,C>0$ such that 
\begin{equation}
c\cdot d\left(P,T_{*}^{n}P\right)\leq-\log\left(\rho\left(P,T_{*}^{n}P\right)\right)\leq C\cdot d\left(P,T_{*}^{n}P\right)\label{eq: Kakutanis observation}
\end{equation}
where 
\[
d\left(\prod P_{i},\prod Q_{i}\right)=\sum_{i\in\ZZ}\left\{ \left(\sqrt{P_{i}(0)}-\sqrt{Q_{i}(0)}\right)^{2}+\left(\sqrt{P_{i}(1)}-\sqrt{Q_{i}(1)}\right)^{2}\right\} .
\]

\end{enumerate}\end{thm}
\begin{proof}
(1) and (2) follow from Kakutani's Theorem, \cite{Kak} on equivalence
of product measures. Parts (3) and (4) are in \cite{Kre}. (5) is
an observation of Kakutani. 
\end{proof}

\subsubsection{The Odometer as the tail action of the shift}

We will also consider the odometer action $\tau$ on $X^{+}$ given
by 
\[
\tau\left(\underset{n-times}{\underbrace{1,1,..,1}},0,w\right)=\left(\underset{n-times}{\underbrace{0,0,..,0}},1,w\right).
\]
The odometer and the one sided shift satisfy 
\[
\Ra_{\tau}=\Tau_{\sigma}.
\]
A calculation shows that 
\[
\tau'(x)=\frac{P_{\phi(x)}(1)}{P_{\phi(x)}(0)}\cdot\prod_{k=1}^{\phi(x)-1}\frac{P_{k}(0)}{P_{k}(1)}
\]
where 
\[
\phi(x):=\min\left\{ n\geq1:\ x_{n}=0\right\} .
\]

The odometer satisfies the so called \textit{Odometer Property, }which
states that for every $N\in\NN$ and $x\in X^{+}$, 
\[
\left\{ \left(\left(\tau^{k}x\right)_{1},\left(\tau^{k}x\right)_{1},\ldots,\left(\tau^{k}x\right)_{N}\right):\ k=0,1,...,2^{N}-1\right\} =\{0,1\}^{N}.
\]
Using this fact one shows that for every $n\in\NN$, 
\begin{equation}
\tau^{\left(2^{n}\right)'}(x)=\tau'\circ\sigma^{n}(x).\label{eq: RN for odometer along rigidity times}
\end{equation}
This can also be deduced from the fact that for all $n\in\NN$ and
$j\in\{1,..,n\}$, 
\[
\left(\tau^{2^{n}}(x)\right)_{j}=x_{j}.
\]
This property plays a crucial role in calculating the essential values
of the odometer action. See the proof of Lemma \ref{lem:conservative shift and odometer}
below.

\subsection{Statement of the main theorem and the Answer to Krengel's question}
\begin{thm}
\label{thm: K and Maharam}\textup{\textcolor{black}{For every $\left(X,\BB,P,T\right)$
a conservative and non singular $K$-shift}}\textup{\textcolor{red}{{}
}}\textup{without an absolutely continuous invariant probability measure
the Maharam extension is a $K$-transformation. }
\end{thm}
As a corollary we get a negative answer to Krengel's question for
non singular $K$-shifts. 
\begin{cor}
\label{cor:type III_1}A conservative, ergodic, $K$$-$non singular
Bernoulli shift is either of type ${\rm III}_{1}$ or type ${\rm II}_{1}$. \end{cor}
\begin{proof}
Assume that there exists no a.c.i.p. By Maharam's theorem, the Maharam
extension is conservative and by Theorem \ref{thm: K and Maharam}
it is $K$ $\sigma$-finite measure preserving transformation. Therefore
by \cite{Par} it is ergodic and so the shift is of type ${\rm III_{1}}$. 
\end{proof}
A non singular transformation is of stable type ${\rm III_{\lambda}}$
if for every ergodic probability preserving transformation $\left(Y,\mathcal{C},\nu,S\right)$
the cartesian product $T\times S$ is of type ${\rm III}_{\lambda}$.
Bowen and Nevo \cite{BN} used actions of stable type ${\rm III_{\lambda}}$
in order to obtain ergodic theorems for measure preserving actions
of countable groups. They ask which groups admit an action of stable
type ${\rm III}_{\lambda}$ with $\lambda>0$. As a corollary of Theorem
\ref{thm: K and Maharam} we get the first examples of such $\ZZ$-actions. 
\begin{cor}
A conservative, ergodic, $K$$-$non singular Bernoulli shift such
that 
\[
\sum_{k=1}^{\infty}\left(P_{k}(0)-\frac{1}{2}\right)^{2}=\infty
\]
is of stable type ${\rm III}_{1}$. \end{cor}
\begin{proof}
Let $\left(Y,\mathcal{C},\nu,S\right)$ be an ergodic probability
preserving transformation, $\tilde{T}$ be the maharam extension of
the shift $T$ and $M_{T\times S}$ denote the Maharam extension of
$T\times S$.

Since the Maharam extension $\tilde{T}$ is conservative then $\tilde{T}\times S$
is conservative. It follows from \cite[Thm 2.7.6 and Corr 3.1.8]{Aa}
that $\tilde{T}\times S$ is ergodic. Since $\tilde{T}\times S=M_{T\times S}$,
it follows that the Maharam extension of $T\times S$ is ergodic and
$T\times S$ is of type ${\rm III}_{1}$. \end{proof}
\begin{rem*}
If $S$ is an infinite measure preserving trasformation such that
$\tilde{T}\times S$ is conservative then $\tilde{T}\times S$ is
ergodic. Thus by Dye's Theorem $\tilde{T}\times S$ is orbit equivalent
to $\tilde{T}$.
\end{rem*}

\subsection{The proof of Theorem \ref{thm: K and Maharam}}

By Theorem \ref{ext:(II.2)-Radon Nykodym derivatives}, the Radon-Nykodym
cocycle $\varphi(x):=\log T^{'}(x)$ is $\BB^{+}$ measurable. 

It follows that the Maharam extension of $T$ is the natural extension
of the skew product\\
 $\left(X^{+}\times\RR,\BB^{+}\otimes\BB_{\RR},P^{+}\otimes e^{s}ds,\sigma_{\varphi}\right).$
Since a transformation is $K$ if and only if it is a natural extension
of an exact transformation, in order to show that the Maharam extension
of the two sided shift is $K$, we will show that the skew product
extension $\sigma_{\log T'}$ is exact. 

This will be done in two steps. First we show that the odometer $\left(X^{+},\BB^{+},P^{+},\tau\right)$
is of type ${\rm III}_{1}$ and then we show that 
\[
\Tau\left(\sigma_{\log T'}\right)=\mathcal{R}\left(\tau_{\log\tau'}\right),
\]
thus the tail action is ergodic. 
\begin{lem}
\label{lem:conservative shift and odometer}Let $P$ be as in \eqref{Half stationary product measure}.
If the shift is conservative and there exists no a.c.i.p then:

(1) There exists a subsequence $\left\{ a_{n_{k}}\right\} $ such
that ${\displaystyle \lim_{k\to\infty}a_{n_{k}}=0}$.

(2) The odometer $\left(X^{+},\BB^{+},P^{+},\tau\right)$ is of type
${\rm III}_{1}$. \end{lem}
\begin{proof}
Denote by 
\[
\mathfrak{A}=\left\{ a\in\mathbb{R}:\ \exists n_{k}\to\infty,\ a_{n_{k}}\xrightarrow[k\to\infty]{}a\right\} 
\]
the set of limit points of the sequence $\left\{ a_{n}\right\} $. 

(1) Assume that $0\notin\mathfrak{A}.$ We will show that then $\sum_{n=1}^{\infty}\rho\left(P,T_{*}^{n}P\right)<\infty$
and so by Lemma \ref{lem:Dissipative zero type} $T$ is dissipative.

Since $0$ is not a limit point of $\left\{ a_{n}\right\} $, there
exists an $\epsilon>0$ and $N\in\NN$ such that for all $i>N$,
\begin{equation}
\sqrt{\frac{1-a_{i}}{2}}-\sqrt{\frac{1}{2}}>\epsilon.\label{eq: if 0 is not a limit point}
\end{equation}
 Therefore for every $n>N$,
\begin{eqnarray*}
d\left(P,T_{*}^{n}P\right) & = & \sum_{i\in\ZZ}\left\{ \left(\sqrt{P_{i}(0)}-\sqrt{P_{i-n}(0)}\right)^{2}+\left(\sqrt{P_{i}(1)}-\sqrt{P_{i-n}(1)}\right)^{2}\right\} \\
 & \geq & \sum_{i=N}^{n}\left(\sqrt{P_{i}(0)}-\sqrt{P_{i-n}(0)}\right)^{2}\\
 & = & \sum_{i=N}^{n}\left(\sqrt{\frac{1-a_{i}}{2}}-\sqrt{\frac{1}{2}}\right)^{2}.
\end{eqnarray*}
The last equality follows from the fact that $P_{k}=\left(\frac{1}{2},\frac{1}{2}\right)$
for $k\in\ZZ\backslash\NN$. Therefore by \eqref{eq: if 0 is not a limit point}
we have that 
\[
d\left(P,T_{*}^{n}P\right)\geq\left(n-N\right)\epsilon^{2}.
\]
The conclusion follows from \eqref{eq: Kakutanis observation} and
Lemma \ref{lem:Dissipative zero type}. 

(2) Let $P$ be a half stationary product measure such that the shift
is conservative and there is no a.c.i.p. 

One can show that we can choose a subsequence such that ${\displaystyle \lim_{n\to\infty}}a_{n_{k}}=0$
and $\sum_{k=1}^{\infty}a_{n_{k}}^{2}=\infty$ and then use standard
techniques. 

Alternatively we can argue as follows: Since there is no a.c.i.p.
then 
\[
\sum_{n=1}^{\infty}a_{n}^{2}=\infty.
\]
Therefore if $\mathfrak{A}=\{0\}$ ($\lim a_{n}=0$) then the odometer
is of type ${\rm III}_{1}$ by \cite[Prop. 3.1.]{DKQ}. 

Otherwise there is $0<\alpha<1$ such that $\left\{ 0,\alpha\right\} \subset\mathfrak{A}$.
It follows from the non-singularity condition \eqref{eq:non singular condition}
that 
\[
[0,\alpha]\subset\mathfrak{A}.
\]
We show that $e\left(\tau,\log\tau'\right)=\RR$ by showing that for
every $p\in\mathfrak{A}\backslash\left\{ -1,1\right\} $, 
\[
\log\frac{1+p}{1-p}\in e\left(\tau,\log\tau'\right),
\]
so the set of essential values contains an interval. This will be
done by establishing the conditions of \cite[Lemma 2.1]{DKQ}. 

Let $p\in\mathfrak{A}$ and $a_{n_{k}}\xrightarrow[k\to\infty]{}p$. 

Let 
\[
C=\left[c\right]_{1}^{n}:=\left\{ x\in X^{+}:\ x_{i}=c_{i}\ \forall i\in[1,n]\right\} .
\]
be a cylinder set and write 
\[
C_{n_{k}}=C\cap\left\{ x\in X^{+}:x_{n_{k}}=0\right\} .
\]
It follows from \eqref{eq: RN for odometer along rigidity times}
that for every $k\in\NN$ such that $n_{k}>n$,
\[
\left.\log\tau^{\left(2^{n_{k}}\right)'}\right|_{C_{n_{k}}}=\log\frac{1+a_{n_{k}}}{1-a_{n_{k}}}.
\]
Therefore 
\begin{eqnarray}
P^{+}\left(C\cap\tau^{-2^{n_{k}}}C\cap\left[\log\tau^{\left(2^{n_{k}}\right)'}=\log\frac{1+a_{n_{k}}}{1-a_{n_{k}}}\right]\right) & \geq & P^{+}\left(C_{n_{k}}\right)\label{eq: return to cylinder sets}\\
 & = & \left(\frac{1-a_{n_{k}}}{2}\right)P^{+}\left(C\right).\nonumber 
\end{eqnarray}
Given $\epsilon>0$, we can choose $k$ large enough such that 
\[
\left(\frac{1-a_{n_{k}}}{2}\right)>\frac{1-p}{4}:=\beta>0.
\]
and 
\[
\left|\log\frac{1+a_{n_{k}}}{1-a_{n_{k}}}-\log\frac{1+p}{1-p}\right|<\epsilon.
\]
Then by \eqref{eq: return to cylinder sets} we get 
\[
P^{+}\left(C\cap\tau^{-2^{n_{k}}}\cap\left[\left|\log\tau^{\left(2^{n_{k}}\right)'}-\log\frac{1+p}{1-p}\right|<\epsilon\right]\right)\geq\beta P^{+}(C).
\]
Thus the conditions of \cite[Lemma 2.1]{DKQ} are satisfied with 
\[
\gamma=\left(\underset{n_{k}-1}{\underbrace{0,..,0},1},\underbar{0}\right)
\]
and 
\[
\mathscr{U}=C_{n_{k}}.
\]
Hence $\log\frac{1+p}{1-p}$ is an essential value for $\log\tau'$. \end{proof}
\begin{lem}
\label{lem: psi is logtau'}Let $P$ be defined by \eqref{Half stationary product measure},
then 
\begin{eqnarray*}
\psi(x) & = & \log\tau'(x),
\end{eqnarray*}
where $\psi(x)$ is the tail-cocycle corresponding to $\varphi=\log T'$.\end{lem}
\begin{proof}
Since 
\[
\sigma^{n}x=\sigma^{n}\tau x\iff n\geq\phi(x)
\]
it follows that
\[
\psi(x)=\sum_{k=0}^{\phi(x)-1}\left\{ \varphi\left(\sigma^{k}x\right)-\varphi\left(\sigma^{k}\tau x\right)\right\} =\varphi_{\phi(x)}(x)-\varphi_{\phi(x)}\left(\tau x\right).
\]
This together with Theorem \ref{thm: Kakutani} and the fact that
\[
\left(\tau x\right)_{k}=\begin{cases}
1-x_{k}, & k\leq\phi(x)\\
x_{k} & k>\phi(x)
\end{cases},
\]
yields 
\begin{eqnarray*}
\psi(x) & = & \log\left(\prod_{k=1}^{\phi(x)}\left[\frac{P_{k-\phi(x)}\left(x_{k}\right)}{P_{k}\left(x_{k}\right)}\left/\frac{P_{k-\phi(x)}\left(1-x_{k}\right)}{P_{k}\left(1-x_{k}\right)}\right.\right]\right)\\
 & = & \log\left(\prod_{k=1}^{\phi(x)}\left[\frac{P_{k-\phi(x)}\left(x_{k}\right)}{P_{k-\phi(x)}\left(1-x_{k}\right)}\cdot\frac{P_{k}\left(1-x_{k}\right)}{P_{k}\left(x_{k}\right)}\right]\right).
\end{eqnarray*}
Since for all $k<0$, $P_{k}\equiv\left(1/2,1/2\right)$, 
\[
\forall k\leq\phi(x),\ \frac{P_{k-\phi(x)}\left(x_{k}\right)}{P_{k-\phi(x)}\left(1-x_{k}\right)}=1,
\]
we see that 
\[
\psi(x)=\log\left(\prod_{k=1}^{\phi(x)}\frac{P_{k}\left(1-x_{k}\right)}{P_{k}\left(x_{k}\right)}\right)=\log\tau'(x).
\]

\end{proof}
\begin{proof}[Proof of Theorem \ref{thm: K and Maharam}] Since the
Maharam extension $\tilde{T}$ is the natural extension of $\sigma_{\varphi}$,
we need to show that $\sigma_{\varphi}$ is exact. 

The odometer $\tau$ is the tail action of the shift $\sigma$. It
follows from Lemma \ref{lem: psi is logtau'} that, 
\[
\Tau\left(\sigma_{\varphi}\right)=\Ra\left(\tau_{\log\tau'}\right).
\]
By Lemma \ref{lem:conservative shift and odometer} $\tau_{\log\tau'}$
is ergodic ( $\tau$ is type ${\rm III}_{1}$) and therefore $\sigma_{\varphi}$
is exact. \end{proof}

\subsection{Countable State space}

By following the same arguments of the previous section, one can show
that if $X=\{1,..,n\}^{\ZZ}$, $T$ is the full shift and $P$ is
a half stationary measure on $X$ which is not equivalent to a stationary
product measure, then the Maharam extension is $K$. 

Consider now the full shift on a countable state space. That is $X=\NN^{\ZZ},$
$T$ is the shift and there exists a proabability measure $p\in\mathcal{P}\left(\NN\right)$
and a sequence $\left\{ p_{j}(\cdot)\right\} _{j=1}^{\infty}$ of
probability measures on $\NN$ so that 
\begin{equation}
P_{k}\left(\cdot\right)=\begin{cases}
p(\cdot), & k\leq0\\
p_{k}\left(\cdot\right), & k>0
\end{cases}.\label{eq:countable half stationary measure}
\end{equation}
The condition for non singularity of the shift becomes now 
\[
d\left(P,T_{*}P\right):=\sum_{k=1}^{\infty}\sum_{j\in\NN}\left(\sqrt{P_{k}(j)}-\sqrt{P_{k-1}(j)}\right)^{2}<\infty.
\]
and it is still true that there exists constants $M,m>0$ so that
for every $n\in\NN$ 
\[
m\cdot d\left(P,T_{*}^{n}P\right)\leq-\log\rho\left(P,T_{*}^{n}P\right)\leq M\cdot d\left(P,T_{*}^{n}P\right).
\]
Therefore we can As before let $\sigma:\NN^{\NN}\circlearrowleft$
be the one sided shift and $P^{+}=\prod_{k\in\NN}P_{k}$. The tail
relation of $\sigma$ is still hyperfinite. For example if we choose
\[
\Tau_{N}=\left\{ (x,y):\ y=x\ {\rm or}\ \exists N\in\NN,\ \forall n>N,x_{n}=y_{n}\ {\rm and\ }\max_{1\leq j\leq N}\left(x_{j},y_{j}\right)\leq N\right\} ,
\]
then $\Tau_{N}$ is an increasing sequence of finite subequivalence
relations with $\cup_{N}\Tau_{N}=\Tau(\sigma)$. However, unlike the
finite state space case, the odometer is no longer the tail action.
In this case it is easier to look at the holonomys of $\Tau(\sigma)$.
An holonomy is a one to one transformations $\phi:Dom(\phi)\to Ran(\phi)$,
here $Dom(\phi),Ran(\phi)\subset X=\NN^{\NN}$, where for every $x\in Dom\mbox{\ensuremath{\left(\phi\right)}}$,
\[
\left(x,\phi(x)\right)\in\Tau\left(\sigma\right).
\]
The ratio set condition for the tail action can be reformulated in
the following way. 

An element $r\in\RR$ is in $R\left(V\right),$ here $\Tau\left(\sigma\right)=\Ra\left(V\right)$,
if for every $A\in\BB_{+}$ and $\epsilon>0$, there exists a $\Tau(\sigma)$
holonomy $\phi$ with $Ran(\phi),Dom(\phi)\subset A$ so that 
\[
\frac{d\phi_{*}P}{dP}=r\pm\epsilon.
\]
For the shift one can generalize Lemma \ref{lem:conservative shift and odometer}
for the countable state case using holonomies of the form $f:[a]_{1}^{n}\to[b]_{1}^{n}$,
\[
f(a_{1},..,a_{n},x):=\left(b_{1},b_{2},..,b_{n},x\right).
\]
In this way one can prove the same result. Either $P\sim\prod p$
or it's Maharam extension is a $K$-transformation.

\section{Examples }

In \cite{Kre,Ham} examples of conservative shifts were constructed
without an a.c.i.p. It follows from Theorem \ref{thm: K and Maharam}
that the Maharam extension is $K$ and that those shifts are of type
${\rm III}_{1}$. In these examples one has 
\begin{equation}
\lim_{n\to\infty}P_{n}(0)=\frac{1}{2}.\label{eq: wrong criteria for conservativity}
\end{equation}
We will give two more examples here. One of a dissipative half stationary
shift with 
\[
\lim_{n\to\infty}P_{n}(0)=\frac{1}{2}
\]
which shows that \eqref{eq: wrong criteria for conservativity} is
not sufficient for conservativity. The other is a conservative half
stationary product measure with 
\[
\liminf_{n\to\infty}P_{k}(0)=\frac{1}{4},\ \limsup_{n\to\infty}P_{k}(0)=\frac{1}{2},
\]
Together those examples show that Lemma \ref{lem:conservative shift and odometer}.1
is all we can say about limit points of $a_{n}$. 
\begin{rem}
Michael Grewe in his Master thesis \cite{Gre} has constructed a different
example of a dissipative shift with $P_{k}(0)\to\frac{1}{2}$. His
method relies on the strong law of large numbers and an inductive
construction. We include here a new example as the method of proof
and the measure are more simple.
\end{rem}

\subsection{Dissipative example. }

Define a product measure by 
\[
P_{n}(0)=\begin{cases}
\frac{1}{2}-\frac{2}{n}, & n\geq2\\
\frac{1}{2} & n<2
\end{cases}.
\]
Since 
\[
\sum_{k=0}^{\infty}\left\{ \left(\sqrt{P_{k}(0)}-\sqrt{P_{k+1}(0)}\right)^{2}+\left(\sqrt{P_{k}(1)}-\sqrt{P_{k+1}(1)}\right)^{2}\right\} <\infty,
\]
the shift $\left(\{0,1\}^{\ZZ},P,T\right)$ is non singular. In addition
\begin{eqnarray*}
d\left(P,P\circ T^{n}\right) & \geq & \sum_{k=0}^{n}\left\{ \left(\sqrt{P_{k}(0)}-\sqrt{P_{k-n}(0)}\right)^{2}+\left(\sqrt{P_{k}(1)}-\sqrt{P_{k-n}(1)}\right)^{2}\right\} \\
 & = & \sum_{k=2}^{n}\left\{ 2-\sqrt{1-\frac{4}{k}}-\sqrt{1+\frac{4}{k}}\right\} .
\end{eqnarray*}
It follows from the Taylor expansion of $\sqrt{1+x}$ that 
\[
2-\sqrt{1-\frac{2}{k}}-\sqrt{1+\frac{2}{k}}=\frac{2\sqrt{2}-1}{k}+O_{k\to\infty}\left(\frac{1}{k^{2}}\right).
\]
Therefore there exists a constant $C\in\RR$ such that 
\[
d\left(P,P\circ T^{n}\right)\geq\left(2\sqrt{2}-1\right)\sum_{k=2}^{n}\frac{1}{k}+C.
\]
Since $ $$\sum_{k=2}^{n}\frac{1}{k}\propto\log(n)$ and $\log\rho\left(P,P\circ T^{n}\right)\propto d\left(P,P\circ T^{n}\right)$,
it follows that 
\[
\sum_{n=1}^{\infty}\rho\left(P,P\circ T^{n}\right)<\infty.
\]
By Lemma \ref{lem:Dissipative zero type} the shift is dissipative.

\subsection{The ``weird'' conservative example}

Given $k\in\NN$ set 
\[
\lambda_{n}^{(k)}=\begin{cases}
1+\frac{n}{2^{k}}, & n\in\left[0,2^{k-1}\right]\\
2-\frac{n}{2^{k}}, & n\in\left[2^{k-1},2\cdot2^{k-1}\right]\\
1, & {\rm otherwise}
\end{cases}.
\]
and let $P^{(k)}$ be the product measure on $X$ with factor measures
\[
P_{n}^{(k)}(1)=\frac{\lambda_{n}^{(k)}}{1+\lambda_{n}^{(k)}}=1-P_{n}^{(k)}(0).
\]

Our example of a conservative product measure with 
\[
\limsup_{k\to\infty}P_{k}(1)=\frac{3}{4}>\frac{1}{2}=\liminf P_{k}(1)
\]
consists of large intervals where $P_{k}(0)$ is exactly $\frac{1}{2}$
followed by large intervals of the form $\left[N,N+2^{k}\right]$
where $P_{n}(1)$ equals $P_{n-N}^{(k)}(1)$ (a slow increase to $\frac{3}{4}$
followed by a small decrease back to $\frac{1}{2}$). Then this segment
is followed by a larger segment where $P_{k}(0)=\frac{1}{2}$ and
so on. The main difficulty in showing that 
\[
\sum T^{n'}=\infty
\]
is in showing that for some $k's$ we have $N(k)$ such that 
\[
T^{k'}(w)\approx\prod_{n=0}^{N(k)}\frac{P_{n-k}\left(w_{n}\right)}{P_{n}\left(w_{n}\right)}
\]
on a set of positive measure. For that purpose we need the following
lemma which states that if $k$ is large enough with respect to $m$
then the derivatives of the shift under the measure $P^{(k)}$ are
bounded from below up to time $m$ on a set of large measure. 
\begin{lem}
\label{lem: P^(k) has small derivatives}Given $m$ and $t$ there
exists a $k\in\NN$ such that 
\[
P^{(k)}\left(\inf_{l\leq m}T_{(k)}^{l'}(w)\geq e^{-2^{-t}}\right)\geq1-2^{-t}.
\]
\end{lem}
\begin{proof}
It follows from \ref{ext:(II.2)-Radon Nykodym derivatives} and the
structure of $P^{(k)}$that for $l<2^{k-1}$, 
\begin{eqnarray*}
\log\left(T_{(k)}^{l'}(w)\right) & = & \log\left(\prod_{n=0}^{2^{k}+l}\frac{P_{n-l}^{(k)}\left(w_{n}\right)}{P_{n}^{(k)}\left(w_{n}\right)}\right)\\
 & = & \log\left(\prod_{n=0}^{2^{k}+l}\left(\frac{\lambda_{n-l}^{(k)}}{\lambda_{n}^{(k)}}\right)^{w_{n}}\right)\\
 &  & \sum_{n=0}^{2^{k}+l}w_{n}\left(\log\lambda_{n-l}^{(k)}-\log\lambda_{n}^{(k)}\right).
\end{eqnarray*}
Using the fact that for every $n<2^{k-1}$, 
\[
\lambda_{n}^{(k)}=\lambda_{2^{k}-n}^{(k)}
\]
and a rearrangement of the sum one has
\begin{equation}
\log\left(T_{(k)}^{l'}(w)\right)=\sum_{n=0}^{2^{k-1}-l}Y_{n,k,l}+f(k,l)(w),\label{eq: form of logT'}
\end{equation}
where 
\[
Y_{n,k,l}:=\left(\log\lambda_{n-l}^{(k)}-\log\lambda_{n}^{(k)}\right)\left(w_{n+l}-w_{2^{k}-n}\right)
\]
 and 
\[
f(k,l)(w)=\left(\sum_{n=2^{k-1}}^{2^{k-1}+l}+\sum_{n=0}^{l}+\sum_{n=2^{k}}^{2^{k+l}}\right)\left[w_{n}\left(\log\lambda_{n-l}^{(k)}-\log\lambda_{n}^{(k)}\right)\right].
\]
By a trivial bound 
\begin{equation}
\left|f(k,l)(w)\right|\leq3l\max_{n\in\NN}\left(\log\lambda_{n-l}^{(k)}-\log\lambda_{n}^{(k)}\right)\leq\frac{3l^{2}}{2^{k}}.\label{eq: first neglibile part}
\end{equation}
To bound the first term notice that
\[
\mathbb{E}_{P^{(k)}}\left(Y_{n,k,l}\right)\propto\frac{l^{2}}{2^{2k}}\ {\rm and}\ \mathrm{Var}_{P^{(k)}}\left(Y_{n,k,l}\right)\propto\frac{l}{2^{3k}}.
\]
By independence of the $Y_{n,k,l}$'s we have 
\[
\mathrm{Var}\left(\sum_{n=0}^{2^{k-1}-l}Y_{n,k,l}\right)\propto\frac{l^{2}}{2^{2k}}\ll\left(\frac{l^{2}}{2^{k}}\right)\propto\mathbb{E}\left(\sum_{n=0}^{2^{k-1}-l}Y_{n,k,l}\right).
\]
It follows from this equation, Equations \eqref{eq: first neglibile part},
\eqref{eq: form of logT'} and Chebyshev's inequality that if $k$
is large enough relative to $m$ and $t$ then for every $l<m$, 
\[
P^{(k)}\left(T_{(k)}^{l'}(w)\leq e^{2^{-t}}\right)\leq\frac{e^{-t}}{m}.
\]
The Lemma follows from a union bound. 
\end{proof}
Now we are ready to construct the product measure. 

Let $\mathrm{P}=\prod\mathrm{P}_{k}$ where for $k\leq0$, 
\[
\mathrm{P}_{k}(0)=\P_{k}(1)=\frac{1}{2}.
\]
To define $\P_{k}$ for positive $k$ we choose inductively two subsequences
$\left\{ n_{t}\right\} _{t\in\NN},\left\{ m_{t}\right\} _{t=0}^{\infty}$
with 
\[
0<n_{t}<m_{t}<n_{t+1}
\]
and $m_{0}=0$. The factor measures will be fair coins for $j\in\left[n_{t},m_{t}\right]$
and on the other segments we will choose them according to $P^{\left(k_{t}\right)}$. 

Definition of $n_{t}$ given $m_{t-1}$ and $\P|_{\left[m_{t-1},n_{t}\right)}$: 

By Lemma \ref{lem: P^(k) has small derivatives} there exists $k_{t}$
such that 
\[
P^{\left(k_{t}\right)}\left(\inf_{l\leq m_{t}}T_{\left(k_{t}\right)}^{l'}(w)\geq e^{-2^{-t}}\right)\geq1-2^{-t}.
\]
Let $n_{t}=m_{t}+2^{k_{t}}$. Now for $m_{t-1}\leq j\leq n_{t}$ set
\[
\P_{j}=P_{j-m_{t-1}}^{\left(k_{t}\right)}.
\]
Definition of $m_{t}$ given $n_{t}$ and $\P|_{\left[n_{t},m_{t}\right)}$:
Let 
\begin{equation}
m_{t}=n_{t}+2^{n_{t}}.\label{eq: choice of m_t}
\end{equation}

For conclusion 
\[
\P_{j}(1)=1-\P_{j}(0)=\begin{cases}
\frac{1}{2}, & j<0\\
P_{j-m_{t-1}}^{\left(k_{t}\right)}, & m_{t-1}\leq j<n_{t}\\
\frac{1}{2}, & n_{t}\leq j<m_{t}
\end{cases}.
\]
The measure satisfies 
\[
\liminf_{k\to\infty}\P_{k}(1)=\frac{1}{2}
\]
and 
\begin{eqnarray*}
\limsup_{k\to\infty}\P_{k}(1) & = & \lim_{k\to\infty}\P_{m_{t-1}+2^{k_{t}-1}}(1)\\
 & = & \lim_{k\to\infty}P^{\left(k_{t}\right)}{}_{2^{k_{t}-1}}(1)=\frac{3}{4}.
\end{eqnarray*}

\begin{prop}
\label{prop:Weird shift}The shift $\left(X,\BB,\P,T\right)$ is conservative
and ergodic and type ${\rm III}_{1}$. 
\end{prop}
Sketch of proof: The first step will be to show that if $m<m_{t}$
then 
\[
T^{m'}(w)\geq\left(\frac{3}{2}\right)^{-n_{t}}\prod_{u=t}^{\infty}T_{\left(k_{t}\right)}^{m'}(w(t))
\]
where $\left\{ w(t)\right\} _{t=1}^{\infty}\subset X$ are random
sequences which are independent of one another and for each $t$,
$w(t)$ is distributed as $P^{\left(k_{t}\right)}$. 

Then we will use Lemma \ref{lem: P^(k) has small derivatives} to
bound $T^{m'}$ for $m\in\left[n_{t},m_{t}\right)$ on a set of positive
measure. This will give us that $\mathfrak{C}\neq\emptyset$ which
by a result of Grewe, see Lemma \ref{lem:[Grewe]}, yields $X=\mathfrak{C}$. 
\begin{lem}
\label{lem: Derivatives in m_t-1,n_t}For every $n_{t}\leq n<m_{t}$,
\[
\frac{d\P\circ T^{n}}{d\P}=T^{n'}(w)=\left(\prod_{k=1}^{t}\prod_{u=m_{k-1}}^{n_{k}-1}\frac{1}{2\P_{u}\left(w_{u}\right)}\right)\cdot\left(\prod_{l=t+1}^{\infty}\prod_{u=m_{l-1}}^{n_{l}+n-1}\frac{\P_{u-n}\left(w_{u}\right)}{\P\left(w_{u}\right)}\right).
\]
\end{lem}
\begin{proof}
This is a combination of the Theorem \ref{thm: Kakutani}.2 and the
fact that for every $k\notin\cup_{k=1}^{\infty}\left[m_{t-1},n_{t}\right)$,
\[
\P_{k}\left(w_{k}\right)\equiv\frac{1}{2}\ \forall w_{k}\in\{0,1\}.
\]
 Note that we also used the fact that for every $l>t$, and $n<m_{t}$
\[
m_{l-1}-n>m_{l-1}-n_{l-1}>n_{l-1}
\]
 so the segments $\left[m_{l-1},n_{l-1}\right)$ do not overlap when
we shift by $n$. 
\end{proof}
\begin{proof}[Proof of Proposition \ref{prop:Weird shift}]

Set 
\[
A_{t}=\left\{ w\in X:\ \forall k\leq m_{t-1}.\ \prod_{u=m_{t}}^{n_{t+1}+k}\frac{\P_{u-k}\left(w_{u}\right)}{\P_{u}\left(w_{u}\right)}\geq e^{-2^{-t}}\right\} .
\]
We have that $A_{1},A_{2},..$ are independent and since 
\[
\prod_{u=m_{t}}^{n_{t}+n}\frac{\P_{u-k}\left(w_{u}\right)}{\P_{u}\left(w_{u}\right)}=\prod_{u=m_{t}}^{n_{t+1}+n}\frac{P_{u-m_{t}-n}^{\left(k_{t}\right)}\left(w_{u}\right)}{P_{u-m_{t}}^{\left(k_{t}\right)}\left(w_{u}\right)}=T_{\left(k_{t+1}\right)}^{n'}\left(w|_{\left[m_{t},n_{t+1}+n\right)}\right)
\]
and $\P|_{\left[m_{t},n_{t+1}+n\right)}=P^{\left(k_{t}\right)}|_{\left[0,2^{k_{t}}+n\right]}$
we have by Lemma \ref{lem: P^(k) has small derivatives} and the choice
of $k_{t}$ that 
\[
\P\left(A_{t}\right)=P^{\left(k_{t+1}\right)}\left(\inf_{k\leq m_{t}}T_{\left(k_{t+1}\right)}^{n'}\geq e^{-2^{-t}}\right)\geq1-e^{-t}.
\]
Set $A=\cap_{t}A_{t}$. Then 
\[
\P\left(A\right)\geq\prod_{t=1}^{\infty}\left(1-e^{-t}\right)>0.
\]
For every $m_{t-1}\leq n\leq m_{t}$, $l>t$ and $w\in A$ we have
\[
\prod_{u=m_{l}}^{n_{l+1}+n}\frac{\P_{u-k}\left(w_{u}\right)}{\P_{u}\left(w_{u}\right)}\geq e^{-2^{-l}}.
\]
Applying the last inequality together with Lemma \ref{lem: Derivatives in m_t-1,n_t}
we see that for $w\in A$ and $n_{t-1}\leq n\leq m_{t}$,
\begin{eqnarray*}
T^{n'}(w) & \geq & \left(\prod_{k=1}^{t-1}\prod_{u=m_{k-1}}^{n_{k}-1}\frac{1}{2\P_{u}\left(w_{u}\right)}\right)\cdot\prod_{j=l}^{\infty}e^{-2^{-l}}\\
 & \geq & e^{-1}\prod_{k=1}^{n_{t}}\frac{1}{2\cdot\frac{3}{2}}=\frac{1}{e}\cdot\left(\frac{2}{3}\right)^{n_{t}}.
\end{eqnarray*}
Therefore for every $w\in A$,
\begin{eqnarray*}
\sum_{n=1}^{\infty}T^{n'}(w) & \geq & \sum_{t=1}^{\infty}\sum_{u=n_{t-1}}^{m_{t}}T^{n'}(w)\\
 & \geq & e^{-1}\sum_{t=1}^{\infty}\left[\left(\frac{2}{3}\right)^{n_{t-1}}\left(m_{t}-n_{t-1}\right)\right]=\infty.
\end{eqnarray*}
Here the last assertion follows from \eqref{eq: choice of m_t}. Thus
$A\subset\mathfrak{C}$. By Lemma \eqref{lem:[Grewe]} the shift is
conservative. 

\end{proof}

\section{Apendix}

Here we give a proof of a result from {[}Grewe{]}. 
\begin{lem}
\label{lem:[Grewe]}{[}Grewe{]} Let $P$ be a product measure on $X$.
Then if the factor measures are bounded away from $0$ and $1$ (e.g.
$\exists p>0$ s.t. $\forall k\in\mathbb{Z}$, $p<P_{k}(0)<1-p$)
then the shift $\left(X,P,T\right)$ is either conservative or dissipative.\end{lem}
\begin{proof}
The condition on the factor measures ensures that for every $k\in\mathbb{Z}$,
$w_{1},x_{1}\in\{0,1\}$ 
\[
c:=\min\left(\frac{p}{1-p},\frac{1-p}{p}\right)\leq\frac{P_{k}\left(x_{1}\right)}{P_{k}\left(w_{1}\right)}\leq c^{-1}.
\]
This means that if $x,w\in\{0,1\}^{\mathbb{Z}}$ defer in only finitely
many coordinates then there exists $M>0$ s.t 
\[
\frac{1}{M}T^{n'}(x)\leq T^{n'}(w)=\prod_{k=1}^{\infty}\frac{P_{k-n}\left(w_{k}\right)}{P_{k}\left(w_{k}\right)}\leq MT^{n'}(x).
\]
Therefore 
\[
\sum_{n=1}^{\infty}T^{n'}(w)=\infty\Leftrightarrow\sum_{n=1}^{\infty}T^{n'}(x)=\infty
\]
and so the conservative and the dissipative parts are in 
\[
\cap\mathcal{F}_{n}
\]
where $\mathcal{F}_{n}$ is the sub sigma algebra generated by $\left\{ w_{k}:|k|\geq n\right\} $.
By the Zero One Law $\mathfrak{C}=X$ or $\mathfrak{D}=X$. \end{proof}


\begin{thebibliography}{ALV}
\bibitem[Aa]{Aa}J. Aaronson. An introduction to infinite ergodic
theory. Mathematical Surveys and Monographs, 50. American Mathematical
Society, Providence, RI, 1997. 

\bibitem[ALV]{ALV}J. Aaronson, M. Lema\'{n}czyk, D.Volný. A cut salad
of cocycles. Dedicated to the memory of Wies\l{}aw Szlenk. Fund. Math.
157 (1998), no. 2-3, 99\textendash{}119.

\bibitem[ANS]{ANS} J. Aaronson, H. Nakada, O. Sarig. Exchangeable
measures for subshifts. Ann. Inst. H. Poincaré Probab. Statist. 42
(2006), no. 6, 727\textendash{}751.

\bibitem[BN]{BN}L. Bowen, A. Nevo. Pointwise ergodic theorems beyond
amenable groups. Ergodic Theory and Dynamical Systems, Available on
CJO doi:10.1017/S0143385712000041.

\bibitem[DKQ]{DKQ}A.H. Dooley, I. Klemeš, A.N. Quas. Product and
Markov measures of type III. J. Austral. Math. Soc. Ser. A 65 (1998),
no. 1, 84\textendash{}110.

\bibitem[Ham]{Ham}T. Hamachi. On a Bernoulli shift with nonidentical
factor measures. Ergodic Theory Dynamical Systems 1 (1981), no. 3,
273\textendash{}283 (1982).

\bibitem[Gre]{Gre}M. Grewe. Über Konservative und Dissipative Transformationen,
Diplomarbeit, Univ. of Göttingen, Inst. Math. Stoch. (1983). 

\bibitem[Kak]{Kak}S. Kakutani. On equivalence of infinite product
measures. Ann. of Math. (2) 49, (1948). 214\textendash{}224.

\bibitem[KM]{KM}A.S. Kechris, B.D Miller. Topics in orbit equivalence.
Lecture Notes in Mathematics, 1852. Springer-Verlag, Berlin, 2004.

\bibitem[Kos]{Kos}Z. Kosloff. On a type III1 Bernoulli shift. Ergodic
Theory and Dynamical Systems, 31 , pp 1727-1743 (2011). 

\bibitem[Kre]{Kre}U. Krengel, Transformations without fi{}nite invariant
measures have fi{}nite strong generators, Contributions to Ergodic
Theory and Probability (Ohio State Univ., Columbus, OH, 1970), Lecture
Notes in Math., no. 160, Springer-verlag, Berlin, 1970, pp. 135\textendash{}157,
MR 42 \#4703, Zbl 201.38303. 

\bibitem[Par]{Par}W. Parry. Ergodic and spectral analysis of certain
infinite measure preserving transformations. Proc. Amer. Math. Soc.
16 1965 960\textendash{}966.

\bibitem[ST]{ST}C.E. Silva, P. Thieullen. A skew product entropy
for nonsingular transformations. J. London Math. Soc. (2) 52 (1995),
no. 3, 497\textendash{}516.

\bibitem[SlS]{SlS}Slaman, T., Steel, J.: Definable functions on degrees.
In: Kechris, A.S., Martin, D.A., Steel, J.R. (ed) Cabal Seminar, 81\textendash{}85
(Lecture Notes in Mathematics 1333). Springer-Verlag, Berlin (1988).

\bibitem[We]{We}Weiss, B.: Measurable Dynamics. Contemp. Math., 26,
395\textendash{}421 (1984)\end{thebibliography}
\end{document}